\documentclass[english,a4paper,12pt,reqno]{article}
\usepackage[T1]{fontenc}
\usepackage[latin9]{inputenc}
\usepackage{babel}
\usepackage{latexsym, amsfonts, amsthm,amssymb,amscd,amsmath}
\usepackage{enumerate}
\usepackage{exscale}
\usepackage{color}
\usepackage{graphicx}
\usepackage{overpic}
\usepackage{bm}
\usepackage{bbm}
\usepackage{mathrsfs}
\usepackage{enumerate}
\usepackage{natbib}

\definecolor{DarkGreen}{rgb}{0.2,0.6,0.2}

\numberwithin{equation}{section}

\def\mb#1{\boldsymbol{\mathbf{#1}}}

\def\<{\langle}\def\>{\rangle}

\newtheorem{theorem}{Theorem}[section]
\newtheorem{proposition}[theorem]{Proposition}
\newtheorem{corollary}[theorem]{Corollary}
\newtheorem{lemma}[theorem]{Lemma}

\theoremstyle{definition}
\newtheorem{definition}[theorem]{Definition}
\newtheorem{example}[theorem]{Example}
\newtheorem{remark}[theorem]{Remark}

\renewcommand{\baselinestretch}{1.1}\normalsize

\title{\bf Constructing functions with prescribed\\ pathwise quadratic variation}
\author{\normalsize Yuliya Mishura\\
\normalsize Department of Probability, Statistics\\
\normalsize  and Actuarial Mathematics\\ \normalsize Taras Shevchenko National University of Kyiv\\
\normalsize  01601 Kyiv, Ukraine\\
\normalsize    {\tt myus@univ.kiev.ua}\and  \normalsize Alexander Schied\thanks{
Support by Deutsche Forschungsgemeinschaft through the Research Training Group RTG 1953 is gratefully acknowledged.}\\  \normalsize Department of Mathematics\\
 \normalsize University of Mannheim\\
 \normalsize 68131 Mannheim, Germany\\
 \normalsize {\tt alex.schied@gmail.com} }

\date{\small First version: November 14, 2015\\
\small This version: April 12, 2016 }

\begin{document}
\setcitestyle{numbers,open={[},close={]}}

\maketitle

\begin{abstract}
We construct rich vector spaces of continuous functions with prescribed curved or linear pathwise quadratic variations. We also construct a class of functions whose quadratic variation may depend in a local and nonlinear way on the function value. These functions can then be used as integrators in F\"ollmer's pathwise It\=o calculus. Our construction of the latter class of functions relies on an extension of the Doss--Sussman method to a class of nonlinear It\=o differential equations for the F\"ollmer integral. As an application, we provide a deterministic variant of the support theorem for diffusions. We also establish that many of  the constructed functions are nowhere differentiable.\end{abstract}


\section{Introduction}

In the seminal paper~\citep{FoellmerIto}, H.~F\"ollmer provided a strictly pathwise approach to It\=o's formula. The formula is  \lq\lq pathwise\rq\rq\ in the sense that integrators are fixed, nonstochastic functions $x:[0,1]\to\mathbb{R}$ that do not need to arise as typical sample paths of a semimartingale. It thus became clear that It\=o's formula is essentially a second-order extension of the fundamental theorem of calculus for Stieltjes integrals. A systematic introduction to pathwise It\=o calculus, including an English translation of~\citep{FoellmerIto}, is provided in~\citep{Sondermann}.

In recent years, there has been an increased interest in pathwise It\=o calculus. On the one hand, this increase is due to a growing sensitivity to model risk in mathematical finance and economics and the ensuing aspiration to construct dynamic trading strategies without reliance on a probabilistic model. As a matter of fact, a number of recent case studies have shown that some nontrivial results of this type can be obtained by means of pathwise It\=o calculus; see, e.g.,~\citep{Benderetal1,BickWillinger,DavisRavalObloij,FoellmerECM,Lyons95,SchiedCPPI,SchiedStadje,Vovk12}. On the other hand, the recent functional extension of F\"ollmer's pathwise It\=o formula by  Dupire~\citep{Dupire} and Cont and Fourni\'e~\citep{ContFournieJFA, ContFournieAOP} facilitated new and exciting mathematical developments such as the theory of viscosity solution of  partial differential equations on   infinite-dimensional path space~\citep{EkrenKellerTouzi,EkrenTouziZhang}.

A function $x\in C[0,1]$ can be used as an integrator in F\"ollmer's pathwise It\=o calculus if it admits a continuous  pathwise quadratic variation $t\mapsto \<x\>_t$ along a given refining sequence  of partitions of $[0,1]$. It is, however, not entirely straightforward to construct  functions with a given, nontrivial quadratic variation. Of course, one can use the sample paths of a continuous semimartingale, but these will satisfy the requirement only in an almost sure sense, and it will not be possible to determine whether a particular sample path   will be as desired or belong to the nullset of trajectories for which the quadratic variation does not exist. Based on results by Gantert~\citep{GantertDiss,Gantert}, a set $\mathscr{X}\subset C[0,1]$ was constructed in~\citep{SchiedTakagi} for which each element $x\in\mathscr{X}$ has the linear quadratic variation $\<x\>_t=t$. This set, however, has the disadvantage that the quadratic variation of the sum of $x,y\in\mathscr{X}$ need not exist. The existence of $\<x+y\>$ is equivalent to the existence of the covariation $\<x,y\>$ and is crucial for multidimensional pathwise It\=o calculus.

In this note, our goal is to construct rich classes of functions  with prescribed pathwise quadratic variation so that these functions can serve as test integrators for pathwise It\=o calculus. More precisely,  we will construct the following three classes of functions.
\begin{enumerate}[\rm(A)]
\item A vector space of functions $x$ with the (curved) quadratic variation $\<x \>_t=\int_0^tf^2(s)\,ds$ for all $t$, where $f$ is a certain Riemann integrable function  associated with $x$.
\item A vector space of functions $y$ with the (linear) quadratic variation $\<y\>_t=t\int_0^1f^2(s)\,ds$ for all $t$,  where $f$ is again a certain Riemann integrable function associated with $y$.
\item A class of functions $z$ with the \lq\lq local\rq\rq\ quadratic variation $\<z\>_t=\int_0^t\sigma^2(s,z(s))\,ds$ for some sufficiently regular \lq\lq volatility\rq\rq\ function $\sigma$.
\end{enumerate}
The class in (C) was first postulated and used by Bick and Willinger~\citep{BickWillinger}; see also~ \citep{Lyons95,SchiedVoloshchenkoArbitrage}. Our corresponding result now establishes   a path-by-path construction of such functions without the need to rely on  selection from the sample paths of a diffusion process.

Note that, unlike in the case of stochastic processes, it is not possible to construct the functions in (A) and (C) from functions with linear quadratic variation via time change, because a time-changed function will generally only admit a quadratic variation with respect to a time-changed  sequence of partitions. By contrast, our construction of the vector spaces in (A) and (B) relies on Proposition~\ref{Stolz prop}, which combines an observation by Gantert~\citep{GantertDiss,Gantert} with the Stolz--Cesaro theorem so as to characterize the existence of quadratic variation along the sequence of dyadic partitions by means of the convergence of certain Riemann sums. This argument yields the set in (A) relatively directly, while the set in (B) requires the additional use of an ergodic shift and Weyl's equidistribution theorem. We also prove that many functions in (A) and (B) are nowhere differentiable. The set in (C) will be constructed by solving a pathwise It\=o differential equation for the F\"ollmer integral by means of the Doss--Sussman method, where the It\=o differential equation is driven by a function from the set in (B). We will then show that the set in (C) is sufficiently rich in the sense that it is dense in $C[0,1]$ and that its members can connect any two points within any given time interval. This latter result can be regarded as a deterministic variant of a support theorem for diffusion processes as in~\citep{StroockVaradhanSupport}.

This paper is structured as follows. Preliminary definitions and results, including the above-mentioned Proposition~\ref{Stolz prop}, are collected in Section~\ref{prelim section}. The sets in (A) and (B) are constructed in Section~\ref{Riemann section}.  The existence and uniqueness theorem for pathwise It\=o differential equations, from which the functions in (C) can be obtained, and the corresponding \lq\lq support theorem\rq\rq\ are stated in Section~\ref{IDE section}. All proofs are deferred to Section~\ref{Proofs section}.

\section{Main results}\label{Main results Section}
\subsection{Preliminaries}\label{prelim section}
{A partition of the interval $[0,1]$ will be a finite set $\mathbb{T}=\{t_0,t_1,\dots,t_n\}$ such that $0=t_0<t_1<\cdots<t_n=1$.}
Now let $(\mathbb{T}_n)_{n\in\mathbb{N}}$ be an increasing sequence of partitions $\mathbb{T}_1\subset\mathbb{T}_2\subset\cdots$ of the interval $[0,1]$ such that the mesh of $\mathbb{T}_n$ tends to zero; such a sequence $(\mathbb{T}_n)_{n\in\mathbb{N}}$ will be called a {refining sequence of partitions}.   A typical example will be the sequence of dyadic partitions,
\begin{equation}\label{dyadic partition}
\mathbb{T}_n:=\{k2^{-n}\,|\,n\in\mathbb{N}, k=0,\dots, 2^n\},\qquad n=0,1,\dots
\end{equation}
 It will be convenient to denote by $s'$  the successor of $s$ in $\mathbb{T}_n$, i.e.,
$$s'=\begin{cases}\min\{t\in\mathbb{T}_n\,|\,t>s\}&\text{if $s<1$,}\\
1&\text{if $s=1$.}
\end{cases}
$$
For  $x\in C[0,1]$ one then defines the sequence
\begin{equation*}
  \<
x\>_{t}^n:= \sum_{s\in\mathbb{T}_n,\, s\le  t}
(x(s')-x(s))^2 .
\end{equation*}
We will say that $x$  admits the  \emph{quadratic variation} $\<x\>_t$ along  $(\mathbb{T}_n)$ and at $t\in[0,1]$ if the limit
\begin{align}\label{quadVarLimit}
\<x\>_{t}:=\lim_{n\uparrow\infty}\<x\>^n_{t}
\end{align}
exists. Since the sequence $\<x\>^n_{t}$ need not be monotone in $n$,  it is not clear \emph{a priori} whether the limit in~\eqref{quadVarLimit} exists   for any fixed $t\in[0,1]$. Moreover, even if the limit exists, it may depend strongly on the particular choice of the underlying sequence of partitions; see, e.g.,~\citep[p. 47]{Freedman} and~\citep[Proposition 2.7]{SchiedTakagi}.  For $x,y\in C[0,1]$, let
\begin{align*}
\<x,y\>_{t}^n:=\sum_{s\in\mathbb{T}_n,\, s\le t}(x(s')-x(s))(y(s')-y(s))
\end{align*}
and observe that
\begin{align}\label{polarization in approx}
\<x,y\>_{t}^n=\frac12\Big(\<x+y\>_{t}^n-\<x\>_{t}^n-\<y\>_{t}^n\Big).
\end{align}
If $x$ and $y$ admit the   quadratic variations $\<x\>_t$ and $\<y\>_t$, then it follows from~\eqref{polarization in approx} that
the \emph{covariation of $x$ and $y$},
\begin{align*}%\label{covariation limit}
\<x,y\>_{t}:=\lim_{n\uparrow\infty} \<x,y\>_{t}^n,
\end{align*}
exists at   $t$ if and only if $\<x+y\>_t$ exists.
This, however, need not be the case even if both $\<x\>_t$ and $\<y\>_t$ exist; see~\citep[Proposition 2.7]{SchiedTakagi} for an example. It follows in particular that the class of all functions $x$ that admit the  quadratic variation $\<x\>_t$  along $(\mathbb{T}_n)_{n\in\mathbb{N}}$ is not a vector space. The main goal of this paper will be to construct  sufficiently large classes of functions  with prescribed quadratic variation and such that $\<x,y\>_t$ exists for all $t$ and  all $x,y$ in this class so that this class is indeed a vector space.

As observed by Gantert~\citep{GantertDiss,Gantert}, the quadratic variation of a function $x\in C[0,1]$ along the sequence of dyadic partitions~\eqref{dyadic partition}
is closely related to its development in terms of the  \emph{Faber--Schauder functions}, which   are defined as
$$e_\emptyset(t):=t,\qquad e_{0,0}(t):=  {\max\{0,\min\{t,1-t\}\}},\qquad e_{n,k}(t):=2^{-n/2}e_{0,0}(2^n t-k)
$$
for $t\in\mathbb{R}$, $n=1,2,\dots$, and $k\in\mathbb{Z}$.
The graph of $e_{n,k}$  looks like a wedge with height $2^{-\frac {n+2}2}$,  width $2^{-n}$, and center at $t=(k+\frac12)2^{-n}$. In particular,  the functions $e_{n,k}$ have disjoint support   for  distinct $k$ and fixed $n$.
  It is well known that every $x\in C[0,1]$ can be uniquely represented  by means of the following uniformly convergent  series,
\begin{align}\label{Faber-Schauder representation of cont funct}
x=x(0)+ {(x(1)-x(0))}e_\emptyset+\sum_{m=0}^\infty\sum_{k=0}^{2^m-1}\theta_{m,k}e_{m,k},
\end{align}
where the coefficients $\theta_{m,k}$ are given as
$$\theta_{m,k}=2^{m/2}\bigg(2x\Big(\frac{2k+1}{2^{m+1}}\Big)-x\Big(\frac k{2^{m}}\Big)-x\Big(\frac{k+1}{2^{m}}\Big)\bigg);
$$
{see, e.g.,~\citep[p.3]{LindenstraussTzafriri}. Since in this note we are only dealing with functions on $[0,1]$, we will only need the Faber--Schauder functions $e_{n,k}$ for $k=0,\dots,2^n-1$ and their domain of definition can be restricted to $[0,1]$.} For $t=1$, the equivalence of conditions (a) and (b) in the following proposition was stated in~\citep[Lemma 1.1 (ii)]{Gantert}.

\begin{proposition}\label{Stolz prop}Let $x\in C[0,1]$ have  Faber--Schauder development~\eqref{Faber-Schauder representation of cont funct} and let $(\mathbb{T}_n)$ be the sequence of dyadic partitions. Then, for $t\in\bigcup_n\mathbb{T}_n$,  the following conditions are equivalent.
\begin{enumerate}[\rm(a)]
\item The quadratic variation $\<x\>_t$   exists.
\item The following limit exists,
$$\ell_1(t):=\lim_{n\uparrow\infty}\frac1{2^n}\sum_{m=0}^{n-1}\sum_{k=0}^{\lfloor (2^m-1)t\rfloor}\theta_{m,k}^2.
$$
\item The following limit exists,
\begin{equation*}%\label{Stolz limit qv}
\ell_2(t):=\lim_{n\uparrow\infty}\frac1{2^n}\sum_{k=0}^{\lfloor (2^n-1)t\rfloor}\theta_{n,k}^2.
\end{equation*}

\end{enumerate}
In this case, we furthermore have
$\<x\>_t=\ell_1(t)=\ell_2(t)$.
\end{proposition}

\begin{remark}\label{covariation remark}
  Let $y\in C[0,1]$ have the Faber--Schauder development
  \begin{equation}\label{y Schauder development}
  y=y(0)+{(y(1)-y(0))}e_\emptyset+\sum_{m=0}^\infty\sum_{k=0}^{2^m-1}\vartheta_{m,k}e_{m,k}.
  \end{equation}
Then it follows from Proposition~\ref{Stolz prop} and polarization~\eqref{polarization in approx} that the covariation $\<x,y\>_t$ exists along the sequence of dyadic partitions and for $t\in\bigcup_n\mathbb{T}_n$ if and only if the limit
 $$\ell(t)=\lim_{n\uparrow\infty}\frac1{2^{n}}\sum_{k=0}^{\lfloor (2^n-1)t\rfloor}\theta_{n,k}\vartheta_{n,k}$$
 exists and that, in this case, $\langle x,y\rangle_t=\ell (t)$.
 \end{remark}

 {We emphasize that the formulas for $\<x\>_t$ and $\<x,y\>_t$ obtained in Proposition~\ref{Stolz prop} and Remark~\ref{covariation remark} are only valid if   quadratic variation is considered along the sequence $(\mathbb{T}_n)$ of dyadic partitions~\eqref{dyadic partition}, because it is naturally related to the Faber--Schauder development of continuous functions. In principle,   it should be possible to obtain similar results also for other wavelet expansions, which would correspond to other sequences of partitions, such as general $p$-adic partitions. Such an analysis is, however, beyond the scope of this paper.}

 \begin{remark} Let $x,y \in C[0,1]$ have Faber--Schauder developments {\eqref{Faber-Schauder representation of cont funct} and~\eqref{y Schauder development}} with coefficients  $\theta_{n,k},\vartheta_{n,k}\in\{-1,+1\}$ for all $n$ and $k$.  Then
  $$\frac1{2^{n}}\sum_{k=0}^{\lfloor (2^n-1)t\rfloor}\theta_{n,k}\vartheta_{n,k}
              =\frac{{\lfloor (2^n-1)t\rfloor+1}}{2^n}-\frac2{2^{n}}\sum_{k=0}^{\lfloor (2^n-1)t\rfloor}1_{\{\theta_{n,k}\neq\vartheta_{n,k}\}}=\frac{{\lfloor (2^n-1)t\rfloor+1}}{2^n}-2\nu_n(t),$$
        where
        $$\nu_n(t):=\frac{1}{2^{n}}\text{card}\big\{0\leq k\leq \lfloor(2^n-1)t\rfloor\,\big|\,\theta_{n,k}\neq\vartheta_{n,k}\big\}$$
        is the frequency of non-coincidence. Since  ${\lfloor (2^n-1)t\rfloor}{2^{-n}}\to t$, it follows that $\<x,y\>_t$ exists if and only if the frequency $\nu_n(t)$ converges to a limit $ \nu(t)\in[0,1]$.
\end{remark}

\subsection{Constructing vector spaces of functions with prescribed curved and linear quadratic variation}\label{Riemann section}

We start by  constructing a vector space of functions with curved quadratic variation along the sequence~\eqref{dyadic partition} of dyadic partitions.  To this end, we let $\mathscr{F}$ denote the class of all sequences $\bm f=(f_n)_{n=0,1,\dots}$ of bounded functions $f_n:[0,1]\to\mathbb{R}$  that converge uniformly to a Riemann integrable function $f_\infty:=\lim_nf_n$.
For $\bm f\in\mathscr{F}$, we define
$$\theta_{n,k}(\bm f):=f_n(k2^{-n})
$$
and
$$x_{\bm f}:=\sum_{n=0}^\infty\sum_{k=0}^{2^n-1}\theta_{n,k}(\bm f)e_{n,k}.
$$
The preceding sum converges absolutely since the coefficients $\theta_{n,k}(\bm f)$ are uniformly bounded. As a matter of fact,  the boundedness of the coefficients implies that $x_{\bm f}$ is even H\"older continuous with exponent $1/2$; see
~\citep[Theorem 1]{Ciesielski}. {It follows in particular that the class
$\{x_{\bm f}\,|\,\bm f\in\mathscr{F}\}$
does not contain the typical sample paths of a continuous semimartingale with nonvanishing quadratic variation. }

\begin{proposition}\label{first Riemann prop}Let $(\mathbb{T}_n)$ be the sequence~\eqref{dyadic partition} of dyadic partitions. Then the following assertions hold.
\begin{enumerate}[\rm(a)]
\item If $\bm f\in\mathscr{F}$, then $x_{\bm f}$ admits the continuous quadratic variation $\<x_{\bm f}\>_t=\int_0^tf^2_\infty(s)\,ds$ for all $t\in[0,1]$.
\item If $\bm f,\bm g\in\mathscr{F}$, then $x_{\bm f}$ and $x_{\bm g}$ admit the continuous covariation $\<x_{\bm f},x_{\bm g}\>_t=\int_0^tf_\infty(s)g_\infty(s)\,ds$ for all $t\in[0,1]$.
\end{enumerate}
In particular, the class
$\{x_{\bm f}\,|\,\bm f\in\mathscr{F}\}$
is a vector space of functions admitting a continuous quadratic variation.
\end{proposition}

{See Figure~\ref{x fig} for an illustration of functions in $ x_{\bm f}$.} We continue with a non-differentiability result.
Recall that, by Lebesgue's criterion~\citep[Theorem 11.33]{RudinPrinciples}, a function $f:[0,1]\to\mathbb{R}$ is Riemann integrable if and only if it is bounded and continuous almost everywhere.

\begin{proposition}\label{nondiff prop}For any $\bm f\in\mathscr{F}$, the function  $x_{\bm f}$ is not differentiable at any continuity point, $t$, of $f_\infty$ for which $f_\infty(t)\neq0$. In particular, $x_{\bm f}$ is not differentiable almost everywhere on $\{f_\infty\neq0\}$.
\end{proposition}

Now we will construct a rich vector space of functions possessing a linear quadratic variation. It was shown in~\citep{SchiedTakagi} that all functions $x$ whose Faber--Schauder coefficients take only the values $\pm1$  have the linear quadratic variation $\<x\>_t=t$ for all $t$, but the corresponding class, $\mathscr{X}$, is not a vector space. For our construction, we let
$$t\text{\,mod\,} 1:=t-\lfloor t\rfloor $$
denote the fractional part of $t\ge0$. For $\bm f\in\mathscr{F}$ and $\alpha>0$, we define
$$\vartheta_{n,k}(\alpha,\bm f):=f_n\big(\alpha k\text{\,mod\,} 1 \big)
$$
and
$$y_\alpha^{\bm f}:=\sum_{n=0}^\infty\sum_{k=0}^{2^n-1}\vartheta_{n,k}(\alpha,\bm f)e_{n,k}.
$$
Again, the preceding sum converges absolutely, as all coefficients $\vartheta_{n,k}(\alpha,\bm f)$ are bounded. Moreover, $y_\alpha^{\bm f}$ is H\"older continuous with exponent $1/2$.

\begin{proposition}\label{Weyl prop}Let $(\mathbb{T}_n)$ be the sequence  of dyadic partitions and  $\alpha>0$ be irrational and fixed. Then the following assertions hold.
\begin{enumerate}[\rm(a)]
\item If $\bm f\in\mathscr{F}$, then $y_\alpha^{\bm f}$ admits the linear quadratic variation $\langle y_\alpha^{\bm f}\rangle_t=t\int_0^1f_\infty^2(s)\,ds$ for $t\in[0,1]$.
\item  If $\bm f,\bm g\in\mathscr{F}$, then  $y_\alpha^{\bm f}$ and $y_\alpha^{\bm g}$  admit the linear covariation $\langle y_\alpha^{\bm f},y_\alpha^{\bm g}\rangle_t=t\int_0^1f_\infty(s)g_\infty(s)\,ds$ for $t\in[0,1]$.
\end{enumerate}
In particular, the class
$\{y_\alpha^{\bm f}\,|\,\bm f\in\mathscr{F}\}$
is a vector space of functions admitting a linear quadratic variation.
\end{proposition}

{An illustration of functions $y_\alpha^{\mathbf{f}}$ for various choices of $\alpha$ and $\mathbf{f}$ is given in Figure~\ref{y fig}. When comparing  Figures~\ref{x fig} and~\ref{y fig}, one can see that the  functions $y_\alpha^{\mathbf{f}}$ exhibit a lower degree of regularity and look more \lq\lq random" than the functions $x_{\mathbf{f}}$. This effect is due to the ergodic behavior of the shift $x\mapsto (x+\alpha)\text{\,mod\,}  1$, which underlies the coefficients $\vartheta_{n,k}(\alpha,\mathbf {f})$.}
The following non-differentiability result can be proved in the  same way as Proposition~\ref{nondiff prop}.

\begin{proposition}Suppose the $\alpha>0$ is irrational and $\bm f\in\mathscr{F}$ is such that $|f_\infty|$ is bounded away from zero. Then $y_\alpha^{\bm f}$ is nowhere differentiable.
\end{proposition}

\begin{figure}
\begin{center}
\begin{minipage}[b]{8cm}
\includegraphics[width=8cm]{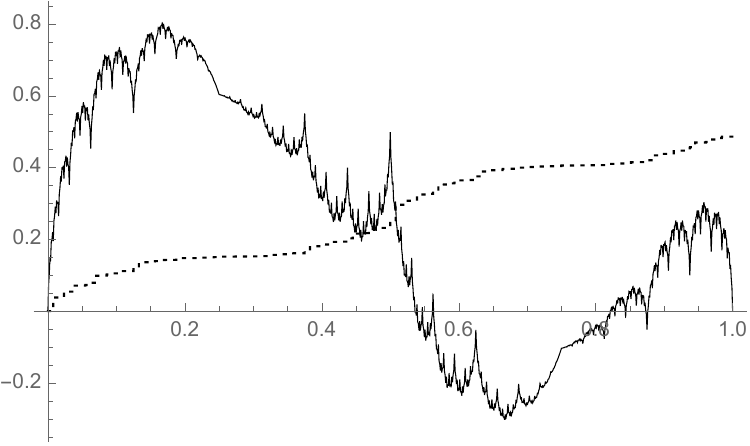}\\
\end{minipage}\qquad
\begin{minipage}[b]{8cm}
\includegraphics[width=8cm]{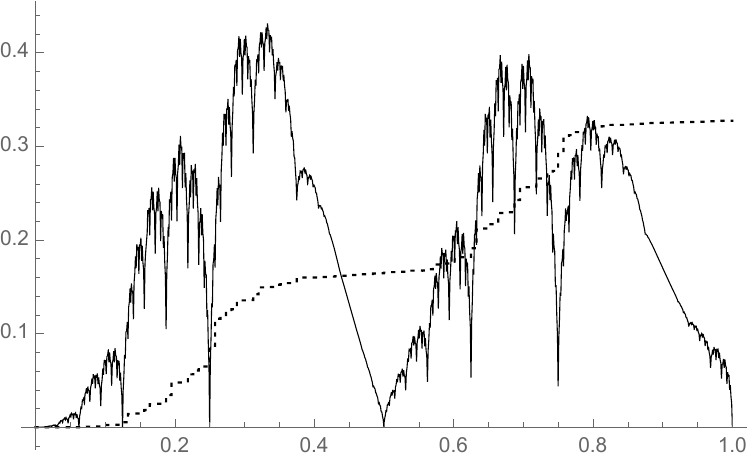}\\
\end{minipage}
\caption{Plots of the functions $x_{\bm f}$ when $f_n(t):=\cos2\pi t$ (left) and $f_n(t):=(\sin7t)^2$ (right) for all $n$. The dotted lines correspond to  $\<x_{\bm f}\>^7$.}
\end{center}\label{x fig}
\end{figure}

\begin{figure}
\begin{center}
\begin{minipage}[b]{8cm}
\includegraphics[width=8cm]{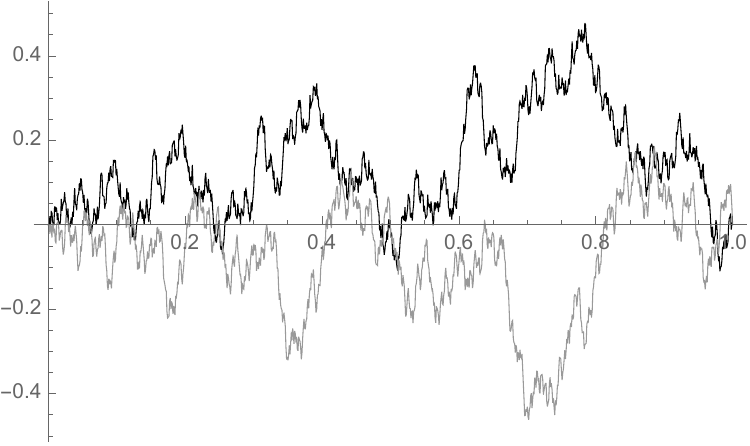}\\
\end{minipage}\qquad
\begin{minipage}[b]{8cm}
\includegraphics[width=8cm]
{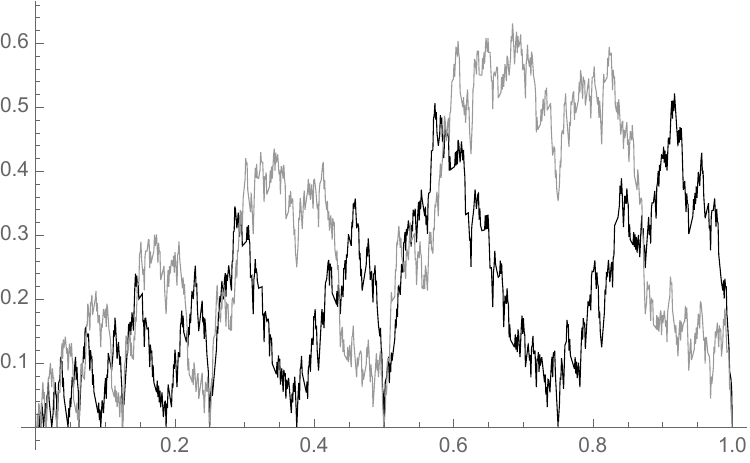}\\
\end{minipage}
\caption{Plots of the functions $y_\alpha^{\bm f}$ when $\alpha=e$ (grey), $\alpha=10e$ (black), $f_n(t):=\sin2\pi t$ (left), and {$f_n(t):=\frac{10t-n}{1+n}\cos\frac{6\pi nt}{1+n}$ (right).}}
\end{center}\label{y fig}
\end{figure}

\subsection{Constructing functions with local quadratic variation via pathwise It\=o differential equations}\label{IDE section}

In this section, $(\mathbb{T}_n)$ may be any refining sequence of partitions; we do not insist that it is given by the dyadic partitions in~\eqref{dyadic partition}.
Bick's and Willinger's approach~\citep{BickWillinger}  to the strictly pathwise hedging of options relies on the following class of trajectories,
\begin{equation}\label{Bick Willinger set}
\Big\{z\in C[0,1]\,\Big|\,\text{$z$ admits the quadratic variation  }\<z\>_t=\int_0^t\beta(s,z(s))\,ds\text{ for all $t$}\Big\}.
\end{equation}
Here,  $\beta$ is a certain strictly positive function, playing the role of a squared local volatility. We will therefore refer to~\eqref{Bick Willinger set} as a set of functions with local quadratic variation. See also~\citep{Lyons95,SchiedStadje,SchiedVoloshchenkoArbitrage} for results involving sets of the form~\eqref{Bick Willinger set}.

In the preceding sections and in~\citep{SchiedTakagi}, trajectories $x$ were constructed that have, e.g., the linear quadratic variation $\<x\>_t=t$. A first guess how to construct a function in the set~\eqref{Bick Willinger set} from a given $x$ with linear quadratic variation  could be to apply a time change as, e.g., in~\citep{EngelbertSchmidt}. This does indeed yield a function with the desired quadratic variation---but a quadratic variation that is taken with respect to a time-changed  sequence of partitions. It is not at all clear if the time-changed function $x$ will also admit a quadratic variation along the original refining sequence $(\mathbb{T}_n)$, and even if it does, it is not clear if it is as desired. Thus, a time change is not an appropriate means of {constructing} functions in~\eqref{Bick Willinger set}. Instead, our approach will be to set up and solve a corresponding It\=o differential equation, whose solution will then belong to the set in~\eqref{Bick Willinger set}.

The discussion of pathwise It\=o differential equations is also interesting in its own right. It is based on F\"ollmer's theory~\citep{FoellmerIto} of pathwise It\=o integration for integrators that admit a continuous quadratic variation. By F\"ollmer's pathwise It\=o formula, the integral $\int_0^t\eta(s)\,dx(s)$ exists as a limit of non-anticipative Riemann sums for all $\eta$ from the    class of admissible integrands, defined below. This pathwise integral  is sometimes called the \emph{F\"ollmer integral}. By $\text{\it CBV}[0,1]$ we will denote the class of continuous functions on $[0,1]$ that are of bounded variation. {The following definition is taken from~\citep[Definition 11]{SchiedCPPI}; see~\citep[Section 3]{SchiedCPPI} for further details.}

\begin{definition}\label{admissible integrand def}Let $x\in C[0,1]$ be a function with continuous quadratic variation along $(\mathbb{T}_n)$.
A function $t\mapsto \eta(t)$ is called an \emph{admissible integrand for $x$} if  there exist $n\in\mathbb{N}$,  a continuous function $\mathbf{A}:[0,1]\to\mathbb{R}^n$ whose components belong to $\text{\it CBV}[0,1]$, an open set  $U\subset \mathbb{R}^n\times\mathbb{R}$ with $(\mathbf{A}(t),x(t))\in U$  for all $t$, and
a continuously differentiable function $f:U\to\mathbb{R}$ for which, {for all $\bm a\in\{\bm b\in\mathbb{R}^n\,|\,\exists \xi\in\mathbb{R}\text{ s.th.~}(\mb b,\xi)\in U\}$,} the function  $ \xi\mapsto f(\mathbf{a},\xi)$ is twice continuously differentiable on its domain, such that $\eta(t)=\frac{\partial}{\partial \xi}f( \mathbf{A}(t),x(t))$.\end{definition}

{\begin{remark}Suppose that $x$ and $\mathbf{A}$ are as in Definition~\ref{admissible integrand def} and that  $g\in C^1(\mathbb{R}^n\times\mathbb{R})$. Then $\eta(t):=g(\mathbf{A}(t), x(t))$ is an admissible integrand for $x$. This follows by taking $f(\mathbf{a},\xi):=\int_0^\xi g(\mathbf{a},y)\,dy$ in Definition~\ref{admissible integrand def}. In particular, $\eta(t)=\exp\{\mu t+\sigma x(t)\}$, $\eta(t)=  x^n(t)$, $ \eta(t)=\exp\{  x^n(t)\}$,  $\eta(t)=\log(1+ x^{2n}(t))$, or smooth functions thereof are admissible integrands for $\mu,\sigma\in\mathbb{R}$ and $n\in\mathbb{N}$.  \end{remark}
}

We can now define the concept of a solution to a pathwise It\=o differential equation.

\begin{definition}\label{IDE def}Suppose that $x\in C[0,1]$ is a function with continuous quadratic variation along $(\mathbb{T}_n)$, $A$ belongs to $\text{\it CBV}[0,1]$, $\sigma:[0,1]\times\mathbb{R}\to\mathbb{R}$ and $b:[0,1]\times\mathbb{R}\to\mathbb{R}$ are continuous functions, and $z_0\in\mathbb{R}$. A function  $z\in C[0,1]$ is called a \emph{solution of the It\=o differential equation}
\begin{equation}\label{IDE def eq}
dz(t)=\sigma(t,z(t))\,dx(t)+b(t,z(t))\,dA(t)
\end{equation}
\emph{with initial condition $z(0)=z_0$} if  $t\mapsto \sigma(t,z(t))$ is an admissible integrand for $x$ and $z$ satisfies the integral form of~\eqref{IDE def eq},
\begin{equation*}%\label{IDE def eq integral form}
z(t)=z_0+\int_0^t\sigma(s,z(s))\,dx(s)+\int_0^tb(s,z(s))\,dA(s),\qquad 0\le t\le 1.
\end{equation*}
\end{definition}

The following result explains why solutions of~\eqref{IDE def eq} provide the desired functions in the class~\eqref{Bick Willinger set} of functions with local quadratic variation. It is an immediate consequence of~\citep[Proposition 12]{SchiedCPPI} and Lemma~\ref{zero qv lemma} below.

\begin{proposition}\label{ide sol qv prop}Suppose that $z$ is a solution of~\eqref{IDE def eq}. Then $z$ has the local quadratic variation
$$\<z\>_t=\int_0^t\sigma^2(s,z(s))\,d\<x\>_s
$$
along  $(\mathbb{T}_n)$.
\end{proposition}

We now extend arguments from Doss~\citep{Doss} and Sussmann~\citep{Sussmann} to show existence and uniqueness of solutions to the It\=o differential equation~\eqref{IDE def eq} when $\sigma$ and $b$ satisfy certain regularity conditions. {As a matter of fact, the same method was used by Klingenh\"ofer and Z\"ahle~\citep{KlingenhoeferZaehle} to construct strictly pathwise solutions of~\eqref{IDE def eq}  when $x$ is H\"older continuous with exponent $\alpha>1/2$ and hence satisfies $\<x\>_t=0$. Our subsequent Theorem~\ref{general IDE existence thm} can thus also be regarded as an extension of~\citep{KlingenhoeferZaehle}  to the case of nonvanishing quadratic variation. In addition to the arguments used in~\citep{Doss,Sussmann,KlingenhoeferZaehle}, we will also need  the associativity theorem for the F\"ollmer integral as established in~\citep[Theorem 13]{SchiedCPPI} and several auxiliary results on nonlinear Stieltjes integral equations, which we have collected in Section~\ref{Stieltjes section}.} As in~\citep{Doss}, the basic idea is to consider the flow $\phi(\tau,\xi ,t)$ associated with $\sigma(\tau,\cdot)$ for fixed $\tau\in[0,1]$, assuming that this flow exists for all $\xi ,t\in\mathbb{R}$, and $\tau\in[0,1]$. That is, $\phi(\tau,\xi ,t)=u(t)$ if $u$ solves the ordinary differential equation $\dot u(t)=\sigma(\tau,u(t))$ with initial condition $u(0)=\xi $. In particular,
\begin{equation}\label{flow t-differential eq}
\phi(\tau,\xi ,0)=\xi \qquad\text{and}\qquad
\phi_t(\tau,\xi ,t)=\sigma(\tau,\phi(\tau,\xi ,t)).
\end{equation}
{Here and and in the sequel, $\phi_t(\tau,\xi,t):=\partial \phi(\tau,\xi,t)/\partial t$, and the partial derivatives $\phi_\tau$, $\phi_{\xi}$, $\phi_{tt}$ etc.~are defined analogously.} We now assume without loss of generality that $x(0)=0$ and define $z$ as
\begin{equation*}%\label{Doss-Sussmann ansatz}
z(t):=\phi(t,B(t),x(t)),
\end{equation*}
where $B\in  CBV[0,1]$ will be determined later. Applying F\"ollmer's pathwise It\=o's formula{, e.g., in the form of~\citep[Theorem 9]{SchiedCPPI},} and using~\eqref{flow t-differential eq} yields
\begin{align*}
z(t)&{=\phi(t,B(t),x(t))}\\
&{=\phi(0,B(0),x(0))+\int_0^t\phi_t(s,B(s),x(s))\,dx(s)+\int_0^t\phi_\tau(s,B(s),x(s))\,ds}\\
&{\qquad+\int_0^t\phi_\xi(s,B(s),x(s))\,dB(s)+\frac12\int_0^t\phi_{tt}(s,B(s),x(s))\,d\<x\>_s}\\
&=B(0)+\int_0^t\sigma(s,z(s))\,dx(s)+\int_0^t\phi_\tau(s,B(s),x(s))\,ds\\&\qquad+\int_0^t\phi_\xi (s,B(s),x(s))\,dB(s)+\frac12\int_0^t\phi_{tt}(s,B(s),x(s))\,d\<x\>_s,
\end{align*}
provided that the function $\phi$ is sufficiently smooth.
So $z$ will solve~\eqref{IDE def eq} if $B$ satisfies the initial condition $B(0)=z_0$ and the sum of three rightmost integrals  agrees with $\int_0^tb(s,z(s))\,dA(s)$.  Both conditions will be satisfied if $B$  solves the following Stieltjes integral equation:
\begin{equation}\label{Doss-Sussman Stieltjes int eq}
\begin{split}
B(t)&=z_0+\int_0^t\frac{b(s,\phi(s,B(s),x(s)))}{\phi_\xi (s,B(s),x(s))}\,dA(s)-\int_0^t\frac{\phi_\tau(s,B(s),x(s))}{\phi_\xi (s,B(s),x(s))}\,ds\\
&\qquad-\frac12\int_0^t\frac{\phi_{tt}(s,B(s),x(s))}{\phi_\xi (s,B(s),x(s))}\,d\<x\>_s.
\end{split}
\end{equation}

{For the sake of precise statements, let us now introduce the following standard terminology. Let $f$ be a real-valued function on $[0,1]\times\mathbb{R}$. We will say that $f$ satisfies a \emph{local Lipschitz condition} if for all $p>0$ there is $L_p\ge0$ such that
\begin{equation}\label{b Lipschitz condition}
 |f(t,\xi)-f(t,\zeta)|\le L_p|\xi-\zeta| \quad \text{ for $\xi,\zeta\in[-p,p]$ and $t\in[0,1]$.}
\end{equation}
Moreover, we will say that $f$ satisfies a \emph{linear growth condition} if
\begin{equation}\label{linear growth condition}
|f(t,\xi)|\le c(1+|\xi|)\qquad\text{for some constant $c\ge0$ and all $t\in[0,1]$, $\xi\in\mathbb{R}$.}
\end{equation}}

\begin{theorem}\label{general IDE existence thm} Suppose that $x\in C[0,1]$  satisfies $x(0)=0$ and admits the continuous quadratic variation $\<x\>$ along $(\mathbb{T}_n)$, $A\in \text{\it CBV}[0,1]$, and $b\in C([0,1]\times\mathbb{R})$ satisfies  both  a local Lipschitz condition~\eqref{b Lipschitz condition} and a linear growth condition~\eqref{linear growth condition}. 
Suppose moreover that there exists some open interval  $I\supset[0,1]$ such that $\sigma(t,\xi)$ is defined for $(t,\xi)\in I\times\mathbb{R}$, belongs to $ C^2(I\times \mathbb{R})$, and has bounded first derivatives in $t$ and $\xi$. Then the flow $\phi(\tau,t,\xi )$ defined in~\eqref{flow t-differential eq} is well-defined for all $\tau\in I$ and $\xi ,t\in\mathbb{R}$, $\phi$ and $\phi_t$ are twice continuously differentiable, the Stieltjes integral equation~\eqref{Doss-Sussman Stieltjes int eq}
 admits a unique solution $B$  for every $z_0\in \mathbb{R}$,  and $z(t):=\phi(t,B(t),x(t))$ is the unique solution of the It\=o differential equation
\begin{equation}\label{general IDE existence thm eq in thm}
dz(t)=\sigma(t,z(t))\,dx(t)+b(t,z(t))\,dA(t)
\end{equation}
with initial condition $z(0)=z_0$.
\end{theorem}

It follows from the preceding theorem and Proposition~\ref{ide sol qv prop} that the solution of~\eqref{general IDE existence thm eq in thm}
 belongs to the set
 \begin{equation*}%\label{ide solution set}
 \Sigma:=\Big\{z\in C[0,1]\,\Big|\,\text{$z$ admits the quadratic variation  }\<z\>_t=\int_0^t\sigma^2(s,z(s))\,{d\<x\>_s}\Big\}.
 \end{equation*}
The following result  implies in particular that {in the case of linear quadratic variation, $\<x\>_t=t$,} the  set $\Sigma$
is dense in $C[0,1]$ and that its members can connect any two points within any given time interval. In this sense, the result can be regarded as a deterministic analogue of a support theorem for diffusion processes as in~\citep{StroockVaradhanSupport}. {The assumption  $\<x\>_t=t$ is not essential and can easily be relaxed. We impose it here because because it is the relevant case for the applications mentioned at the beginning of this section and because it allows us to base the proof of Corollary~\ref{support corollary} on standard results from the theory of ordinary differential equations.  It is also not difficult to prove variants of Corollary~\ref{support corollary} for functions $x$ with general quadratic variation and other drift terms in~\eqref{z dense IDE 1} and~\eqref{z dense IDE}.}

\begin{corollary}\label{support corollary}Suppose that $\sigma$ satisfies the conditions of Theorem~\ref{general IDE existence thm}. Let moreover   $x\in C[0,1]$  be a fixed function that satisfies $x(0)=0$ and admits the linear quadratic variation $\<x\>_t=t$ along $(\mathbb{T}_n)$. Then the following two assertions hold.
\begin{enumerate}[\rm(a)]
\item Let $z_0,z_1\in\mathbb{R}$ and $t_0\in(0,1]$ be given. Then there exists $b\in\mathbb{R}$ such that the solution $z$ of the It\=o differential equation
\begin{equation}\label{z dense IDE 1}
dz(t)=\sigma(t,z(t))\,dx(t)+b\,dt,\qquad z(0)=z_0,
\end{equation}
satisfies $z(t_0)=z_1$.
\item The set of solutions to the It\=o differential equations
\begin{equation}\label{z dense IDE}
dz(t)=\sigma(t,z(t))\,dx(t)+b(t)\,dt
\end{equation}
where $b(\cdot)$ ranges over $C[0,1]$ and $x$ is fixed is dense in $C[0,1]$.
\end{enumerate}
{In particular, the  set~\eqref{Bick Willinger set}
is dense in $C[0,1]$ and  its members can connect any two points within any given nondegenerate time interval.}
\end{corollary}

 Let us now give three concrete examples for Theorem~\ref{general IDE existence thm}. We start with two elementary   cases of    linear It\=o differential equations. The general linear It\=o differential equation can be solved in a similar manner. Our third example is a nonlinear It\=o differential  equation. 

{ \begin{example}   We take $A(t)=t$ and $x\in C[0,1]$  satisfying  $x(0)=0$ and  $\<x\>_t=t$ along $(\mathbb{T}_n)$.  \begin{enumerate}
 \item {\bf (Langevin equation).} We consider the It\=o differential equation
 \begin{equation}\label{langevin eq}
 dz(t)=\sigma\,dx(t)+b z(t)\,dt,\qquad z(0)=z_0,
 \end{equation}
  where $\sigma$ and $b$ are real constants. The corresponding flow $\phi(\tau,\xi,t)$ is independent of $\tau$ and has the form $\phi(\xi,t)=\xi+\sigma t$. The equation~\eqref{Doss-Sussman Stieltjes int eq} is a linear ordinary differential equation (ODE) with   unique solution
 $$B(t)=z_0e^{b t}+\sigma b \int_0^te^{b (t-s)}x(s)\,ds.$$
 Therefore, the unique solution of~\eqref{langevin eq}
 has the form
 $$z(t)=z_0e^{bt}+\sigma b\int_0^te^{b(t-s)}x(s)\,ds+\sigma x(t).$$
\item {(\bf Time-inhomogeneous Black--Scholes dynamics).} Consider the following linear It\=o differential equation
\begin{equation}\label{BS eq}
dz(t)=\sigma(t)z(t)\,dx(t)+b(t)z(t)\,dt,\qquad z(0)=z0,
\end{equation}
where $b\in C[0,1]$ and $\sigma\in C^2(I)$ for some open interval $I\supset[0,1]$.
The corresponding flow  has the  form $\phi(\tau,\xi,t)=\xi e^{\sigma(\tau) t}$. The equation~\eqref{Doss-Sussman Stieltjes int eq}
 becomes equivalent to the ODE $B'(t)=(b(t)-\sigma'(t)x(t)-\frac12\sigma^2(t))B(t)$ with initial condition $B(0)=z_0$, whence
 $$B(t)=z_0\exp\bigg(\int_0^t\Big(b(s)-\sigma'(s)x(s)-\frac{1}{2}\sigma^2(s)\Big)\, ds\bigg).$$
Therefore, the unique solution of~\eqref{BS eq} is
  \begin{equation*}\begin{gathered}z(t)=\phi(t,B(t), x(t))=B(t)e^{\sigma(t)x(t)}\\=
  z_0\exp\bigg(\sigma(t)x(t)+\int_0^t\Big(b(s)-\sigma'(s)x(s)-\frac{1}{2}\sigma^2(s)\Big)\,ds\bigg).\end{gathered}\end{equation*}
When using the fact that
$$\sigma(t)x(t)=\int_0^t\sigma(s)\,dx(s)+\int_0^t\sigma'(s)x(s)\,ds,
$$
which follows from the assumption $x(0)=0$ and F\"{o}llmer's pathwise It\=o formula applied to the function $f(t,x(t))=\sigma(t)x(t)$, we arrive at the more common representation
 $$z(t)=  z_0\exp\bigg(\int_0^t\sigma(s)\,dx(s)+\int_0^t\Big(b(s)-\frac{1}{2}\sigma^2(s)\Big)\,ds\bigg).$$
   \item {\bf(A square-root equation)} The function $\sigma(\xi):=\sqrt{1+\xi^2}$ clearly satisfies the conditions of Theorem~\ref{general IDE existence thm}. The corresponding flow is given by
$\phi(t,\xi)=\sinh\big(t+\sinh^{-1}(\xi)\big)$.
For given drift term $b$, the equation~\eqref{Doss-Sussman Stieltjes int eq} implies the following ordinary differential equation,
\begin{equation}\label{sqrt diffusion B eq}
B'(t)=\frac{\sqrt{1+B(t)^2}}{\cosh\big(x(t)+\sinh^{-1}(B(t))\big)}\Big(b(t,\phi(B(t),x(t))-\frac12\phi(B(t),x(t))\Big).
\end{equation}
Its right-hand side vanishes for $b(t,\xi)=\frac12\xi$, and so $z(t):=\sinh(x(t)+\sinh^{-1}z_0)$ solves
$$dz(t)=\sqrt{1+z(t)^2}\,dx(t)+\frac12 z(t)\,dt,\qquad z(0)=z_0.
$$
Solutions for other choices of $b$ can be obtained by solving~\eqref{sqrt diffusion B eq} numerically.
\end{enumerate}\end{example}
}

\section{Proofs}\label{Proofs section}

\subsection{{Two auxiliary results and the proof of Proposition~\ref{Stolz prop}}}

For the proof of Proposition~\ref{Stolz prop}, we will need two auxiliary lemmas. The first is a simple converse to the Stolz--Cesaro theorem~\citep[Theorem 1.22]{Muresan}.

\begin{lemma}\label{Stolz-Cesaro converse lemma}
Let $(a_n)$ and $(b_n)$ be real sequences such that $b_n>0$, $b_{n+1}/b_n\to \beta\neq1$, and $a_n/b_n\to \ell$. Then also
$$\frac{a_{n+1}-a_n}{b_{n+1}-b_n}\longrightarrow \ell.
$$
\end{lemma}

\begin{proof}We may write
\begin{align*}
\frac{a_{n+1}-a_n}{b_{n+1}-b_n}&=\frac1{\frac{b_{n+1}}{b_n}-1}
\Big(\frac{a_{n+1}}{b_{n+1}}\cdot\frac{b_{n+1}}{b_n}-\frac{a_n}{b_n}\Big).
\end{align*}
Sending $n$ to infinity and using our assumptions thus gives the result.
\end{proof}

The following lemma  can easily be deduced from Propositions 2.2.2, 2.2.9, and 2.3.2 in~\citep{Sondermann}.

\begin{lemma}\label{zero qv lemma} Let $f\in C[0,1]$ be such that $\<f\>_1=0$. Then, for   $x\in C[0,1]$ and $t\in[0,1]$, the quadratic variation $\<x\>_t$ exists  if and only if $\<x+f\>_t$ exists. In this case, we have $\<x\>_t=\<x+f\>_t$  and $\<x,f\>_t=0$.
\end{lemma}

Note that we have $\<f\>_1=0$ whenever $f$ is continuous and of bounded variation.

\begin{proof}[Proof of Proposition~\ref{Stolz prop}] We show first that (a) and (b) are equivalent. To this end, we may assume without loss of generality that $x(0)=x(1)$. Indeed,  the function $f(s):=-x(0)-sx(1)$ is clearly of bounded variation and hence satisfies $\<f\>_1=0$, so that  Lemma~\ref{zero qv lemma} justifies our assumption.

Next, we let $x^t(s):=x(t\wedge s)$,
$$\widetilde\theta_{n,k}:=\begin{cases}\theta_{n,k}&\text{if $k\le \lfloor (2^n-1)t\rfloor$,}\\
0&\text{otherwise,}
\end{cases}
$$
and
$$\widetilde x:=\sum_{m=0}^\infty\sum_{k=0}^{ 2^m-1}\widetilde\theta_{m,k}e_{m,k}.
$$
Since $t\in \bigcup_n\mathbb{T}_n$, the two functions $x^t$ and $\widetilde x$ differ only by a piecewise linear function, $f$, which hence satisfies $\<f\>_1=0$.  Lemma~\ref{zero qv lemma} therefore yields that $\<x\>_t=\<x^t\>_1=\<\widetilde x\>_1$.
Furthermore, it was stated in~\citep[Lemma 1.1 (ii)]{Gantert}  that
$$\<\widetilde x\>_1^n={\frac1{2^n}\sum_{m=0}^{n-1}}\sum_{k=0}^{ 2^m-1}\widetilde\theta_{m,k}^2=\frac1{2^n}\sum_{m=0}^{n-1}\sum_{k=0}^{\lfloor (2^m-1)t\rfloor}\theta_{m,k}^2.
$$
This proves the equivalence of (a) and (b).

 Now we prove the equivalence of (b) and (c). To this end, we let
$$a_n:=\sum_{m=0}^{n-1}\sum_{k=0}^{\lfloor (2^m-1)t\rfloor}\theta_{m,k}^2\qquad\text{and} \qquad b_n:=2^n.
$$
The existence of the limit  in (b) means that $a_n/b_n$ converges to $\ell_1(t)$, whereas the existence of the limit in (c) is equivalent to the convergence of $(a_{n+1}-a_n)/(b_{n+1}-b_n)$ to $\ell_2(t)$. The latter convergence implies the convergence of $a_n/b_n$ to $\ell_2(t)$ by means of the Stolz--Cesaro theorem in the form of~\citep[Theorem 1.22]{Muresan}. On the other hand, the convergence of $a_n/b_n$  to $\ell_1(t)$ entails also the convergence of $(a_{n+1}-a_n)/(b_{n+1}-b_n)$ to $\ell_1(t)$ by way of Lemma~\ref{Stolz-Cesaro converse lemma}. This concludes the proof.
\end{proof}

\subsection{Proofs of results from Section~\ref{Riemann section}}

\begin{proof}[Proof of Proposition~\ref{first Riemann prop}] Note first that the class of Riemann integrable {functions} is clearly an algebra, as can, e.g., be seen from Lebesgue's criterion for Riemann integrability~\citep[Theorem 11.33]{RudinPrinciples}. Thus, $f^2_\infty$ is   Riemann integrable.

{Next, due to the continuity of the function $t\mapsto\int_0^tf_\infty^2(s)\,ds$ and the monotonicity of $t\mapsto\<x_{\bm f}\>_t^n$, it is enough to prove the assertion for $t\in\bigcup_n\mathbb{T}_n$.} Now let $\varepsilon>0$ be given, and take $n_0\in\mathbb{N}$ such that $|f_n(s)-f_\infty(s)|<\varepsilon$ for all $s\in [0,1]$ and  $n\ge n_0$. Then, {for all $n\ge n_0$,}
 \begin{equation}\label{Riemann approx sum}
 \frac1{2^n}\sum_{k=0}^{\lfloor (2^n-1)t\rfloor}\theta_{n,k}(\bm f)^2=\frac1{2^n}\sum_{k=0}^{\lfloor (2^n-1)t\rfloor}f_n(k2^{-n})^2
 \end{equation}
 has a distance of at most $\varepsilon$ to a Riemann sum for $\int_0^tf_\infty^2(s)\,ds$. It follows that the sums in~\eqref{Riemann approx sum} converge to $\int_0^tf_\infty^2(s)\,ds$, and so part (a) of the assertion follows from Proposition~\ref{Stolz prop}. Part (b) follows by polarization as in Remark~\ref{covariation remark}. \end{proof}

\begin{proof}[Proof of Proposition~\ref{nondiff prop}] Our proof uses ideas from~\citep{deRham}, where the non-differentiability of the classical Takagi function was shown.  Let $t\in[0,1)$ be a continuity point of $f_\infty$ such that $f_\infty(t)\neq0$ (the case $t=1$ can be reduced to the case $t=0$ by symmetry). Then there exists $\varepsilon,\delta>0$ such that  $|f_\infty(s)|\ge2\varepsilon$ if $|s-t|\le\delta$. It follows that there exists $n_0\in\mathbb{N}$ such that  $|f_n(s)|\ge\varepsilon$ if $|s-t|\le\delta$ and $n\ge n_0$. For $n\in\mathbb{N}$, we denote by $s_n$ the largest $s\in\mathbb{T}_n$ such that $s\le t$. Its successor, $s'_n$, will then satisfy $s_n'>t$, and we clearly have $s_n,s'_n\to t$ as $n\uparrow\infty$. In particular, $|f_n(s_n)|\ge\varepsilon$ and $|f_n(s_n')|\ge\varepsilon$ if $n\ge n_1:=n_0\vee\lceil-\log_2\delta\rceil$.  We write $x:=x_{\bm f}$ and denote
$$x^n:=\sum_{m=0}^{n-1}\sum_{k=0}^{ 2^m-1}\theta_{m,k}(\bm f)e_{m,k}.
$$
Let us assume by way of contradiction that $x$ is differentiable at $t$. Then we  must have that
$$d_n:=\frac{x({s_n'})-x({s_n})}{s'_n-s_n}=2^n(x({s_n'})-x({s_n}))
$$
converges to a finite limit. Since $e_{p,k}(s_n)=e_{p,k}(s_n')=0$ for $p\ge n$, we have
$$x({s_n'})-x({s_n})=x^n(s_n')-x^n(s_n)=x^{n-1}({s_n'})-x^{n-1}({s_n})+f_n(s_n)\Delta_n,
$$
where   $\Delta_n=2^{-(n+1)/2}$ is the maximal amplitude of a Faber--Schauder function $e_{n-1,k}$. Now we note that
$$x^{n-1}({s_n'})-x^{n-1}({s_n})=\frac12\big(x^{n-1}({s_{n-1}'})-x^{n-1}({s_{n-1}})\big),
$$
because $s_n'-s_n=\frac12(s'_{n-1}-s_{n-1})$, the interval $[s_n,s'_n]$ is contained in $[s_{n-1},s_{n-1}']$, and each Schauder function {$e_{m,k}$ with $m\le n-2$} is linear on the latter interval. We thus arrive at the recursive relation
\begin{equation*}\begin{gathered} d_n=2^n\big(x({s_n'})-x({s_n})\big)=2^{n-1}\big(x^{n-1}({s_{n-1}'})-x^{n-1}({s_{n-1}})\big)+2^nf_n(s_n)\Delta_n\\=d_{n-1}+f_n(s_n) 2^{(n-1)/2}.
\end{gathered}\end{equation*}
Hence, $|d_n-d_{n-1}|\ge \varepsilon 2^{(n-1)/2}$, which contradicts the convergence of  the sequence $(d_n)$.
\end{proof}

\begin{proof}[Proof of Proposition~\ref{Weyl prop}] (a) {As in the proof of  Proposition~\ref{first Riemann prop}, it is enough to prove our formula for $\<y^{\bm f}_\alpha\>_t$ for the case in which $t\in\bigcup_n\mathbb{T}_n$.} By Proposition~\ref{Stolz prop}, we need to investigate the limiting behavior of
$$\frac1{2^n}\sum_{k=0}^{{\lfloor (2^m-1)t\rfloor}}\vartheta_{n,k}^2(\alpha,\bm f)=\frac{{\lfloor (2^m-1)t\rfloor+1}}{2^n}\cdot\frac1{{\lfloor (2^m-1)t\rfloor+1}}\sum_{k=0}^{{\lfloor (2^m-1)t\rfloor}}f^2_n\big(\alpha k\,\text{mod}\,1 \big).
$$
We clearly have {$(\lfloor (2^m-1)t\rfloor+1)2^{-n}\to t$}.  Moreover, since $f^2_\infty$ is Riemann integrable, Weyl's equidistribution theorem~\citep[p.~3]{KuipersNiederreiter} states that
$$\frac1n\sum_{k=0}^{n-1}f^2_\infty\big(\alpha k\,\text{mod}\,1 \big)\longrightarrow\int_0^1f_\infty^2(s)\,ds.$$
The result thus follows from Proposition~\ref{Stolz prop} and by using the uniform convergence $f_n\to f_\infty$. Part
(b) follows as in Proposition~\ref{first Riemann prop}.\end{proof}

\subsection{An auxiliary result  on Stieltjes integral equations}\label{Stieltjes section}

 In this section, we state and prove   an auxiliary result on Stieltjes integral equations, which is needed for the proof of Theorem~\ref{general IDE existence thm}.
Without doubt, this result is well known, but    we have not found a reference  for exactly the version that we need, and so we include it here for the convenience of the reader. It is also not difficult to formulate and prove extensions of this result to the case  in which  both drivers and solutions are multidimensional and possess discontinuities. For the sake of simplicity, however, we confine ourselves to continuous, {though $d$-dimensional,} drivers  and one-dimensional  solutions, as needed for the proof of Theorem~\ref{general IDE existence thm}.

\begin{proposition}\label{Peano prop}Suppose that ${A_1,\dots, A_d}\in \text{\it CBV}[0,1]$,    $ f\in C[0,1]$, and  ${g_1,\dots, g_d}\in C([0,1]\times\mathbb{R}) $ {satisfy the linear growth condition~\eqref{linear growth condition}.}   Then there exists at least one solution $B\in C[0,1]$ of the following   Stieltjes integral equation,
\begin{equation}\label{Peano Stieljes IE}
B(t)=f (t)+ {\sum_{i=1}^d\int_0^tg_i (s,  B(s))\,dA_i(s)},\qquad 0\le t\le 1.\end{equation}
Moreover, the solution~\eqref{Peano Stieljes IE} is unique if $g_1,\dots,g_d$ satisfy in addition a  local Lipschitz condition {as in~\eqref{b Lipschitz condition}.}
\end{proposition}

\begin{proof}
Consider the following Tonelli sequence $( B^{(n)})_{n\in\mathbb{N}}$,\begin{equation}\label{Tonelli Peano Stieljes IE}
B^{(n)}(t)=\begin{cases}f(t)&\text{if $t\in[0,\frac1n]$,}\\\displaystyle
f(t)+ {\sum_{i=1}^d}\int_0^{t-1/n}g_i (s,  B^{(n)}(s))\,dA_i(s) &\text{if $t\in( \frac1n,1]$.}
\end{cases}
\end{equation}
Clearly, the solution $  B^{(n)}$ to~\eqref{Tonelli Peano Stieljes IE}  can be constructed inductively on each interval $(\frac kn,\frac{k+1}n]$.
As in the proof of the classical Peano theorem, the idea is to show that the sequence $(  B^{(n)})_{n\in\mathbb{N}}$  has  an accumulation point with respect to uniform convergence in $C[0,1]$. To this end, we show first that $ B^{(n)}(t)$  is bounded uniformly   in $t$ and $n$. Let $m$ be an upper bound for $|f(t)|$, $t\in[0,1]$,   and let
{$V_i(t)$ denote the  total variation of $A_i$ on $[0,t]$ and $V(t):=\sum_{i=1}^dV_i(t)$}. Let moreover $c\ge0$ be such that {$|g_i(t,y)|\le c(1+|y|)$ for all $i$, $t$, and $y$}. Then,  by a standard estimate for Riemann--Stieltjes integrals (e.g., Theorem 5b on p.~8 of~\citep{WidderLaplaceTransform}),
$$| B^{(n)}(t)| \le m+\int_0^{t-1/n}c(1+|  B^{(n)}(s)|)\,dV(s)\le m+cV(1)+c\int_0^t|  B^{(n)}(s)|\,dV(s).
$$
 Groh's generalized Gronwall inequality~\citep{Groh} (see also Theorem 5.1 in Appendix~5 of~\citep{EthierKurtz})  yields
$$|  B^{(n)}(t)|\le\big(m+cV(1)\big)e^{cV(t)}\le \big(m+cV(1))\big)e^{cV(1)}=:M
$$
for all $t\in[0,1]$. Hence  $(  B^{(n)}(t))_{n\in\mathbb{N}}$ is indeed uniformly bounded in $n$ and $t$. In the next step we show that it is also uniformly equicontinuous. To this end, let
$$K:=\max\Big\{|g_i (t, y)|\,\Big|\,{i=1,\dots,d},\ t\in[0,1],\ |  y|\le M\Big\}.
$$
Then one sees as above that, for $0\le t\le u\le 1$,
\begin{equation*}%\label{Euler scheme bound eq}
|  B^{(n)}(u)-  B^{(n)}(t)|\le |  f(u)-  f(t)|+K\Big(V\big({\max\{0,u-1/n\}}\big)-V\big({\max\{0,t-1/n\}}\big)\Big).
\end{equation*}
Since $V$ is continuous~\citep[Theorem I.3b]{WidderLaplaceTransform}, and hence uniformly continuous, it follows that
 $( B^{(n)})_{n\in\mathbb{N}}$ is indeed uniformly equicontinuous. The Arzela--Ascoli theorem therefore implies the existence of a subsequence $(  B^{(n_k)})_{k\in\mathbb{N}}$ that converges uniformly toward some continuous limiting function $  B$.  The continuity of the Riemann--Stieltjes integral with respect to uniform convergence of integrands  thus  yields that $  B$ solves~\eqref{Peano Stieljes IE}.

Now suppose that {$g_1,\dots,g_d$} satisfy  local Lipschitz conditions {in the form of~\eqref{b Lipschitz condition} and let $L_p$ be the maximum of the corresponding Lipschitz constants for given $p>0$. Next,}    let $  B$ and $\widetilde  B$ be two solutions of~\eqref{Peano Stieljes IE}. Then there exists $p>0$ such that both $  B$ and $\widetilde B$ take values in $[-p,p] $.  Using again the above-mentioned standard estimate for Riemann--Stieltjes integrals  yields
\begin{align*}
|B(t)-\widetilde B(t)|&=\bigg|{\sum_{i=1}^d} \int_0^t\big(g_i (s,  B(s))-g_i (s,  \widetilde B(s))\big)\,dA_i(s)\bigg|\\
&\le {\sum_{i=1}^d}\int_0^t\big|g_i (s,  B(s))-g (s, \widetilde  B(s))\big|\,dV_i(s)\\
&\le{ L_p}\int_0^t|  B(s)-  \widetilde B(s)|\,dV(s).
\end{align*}
The generalized Gronwall inequality that was cited above now yields $ B=\widetilde B$.
  \end{proof}

\subsection{Proof of the results from Section~\ref{IDE section}}

Our proof of Theorem~\ref{general IDE existence thm} follows along the lines of~\citep{Doss}, but several supplementary  arguments are needed because of the time dependence of $\sigma$, the fact that $x$ is not a typical Brownian sample path, and because $A$ is not linear. We will also need the associativity property of the F\"ollmer integral that was established in~\citep[Theorem 13]{SchiedCPPI}. We first collect some properties of the flow $\phi$ in the following lemma. {Throughout this section, we will use the notation introduced in Theorem~\ref{general IDE existence thm}. Recall in particular that $I$ denotes an open interval containing $[0,1]$.}

\medskip

\begin{lemma}\label{flow lemma} Under the assumptions of Theorem~\ref{general IDE existence thm}, the following assertions hold for all $\tau\in I$ and $\xi ,s,t\in\mathbb{R}$.
\begin{enumerate}
\item $\phi(\tau,\xi ,t)$ is well-defined for all $\tau\in I$ and $\xi ,t\in\mathbb{R}$.\label{flow lemma1}
\item $\phi\in C^2(I\times\mathbb{R}\times\mathbb{R})$ and $\phi_t\in C^2(I\times\mathbb{R}\times\mathbb{R})$.\label{flow lemma2}
\item $\phi(\tau,\phi(\tau,\xi ,s),t)=\phi(\tau,\xi ,s+t)$. \label{flow lemma3}
\item $\phi_t(\tau,\xi ,t)=\sigma(\tau,\phi(\tau,\xi ,t))$.\label{flow lemma4}
\item $\phi_{tt}(\tau,\xi ,t)=\sigma_\xi (\tau,\phi(\tau,\xi ,t))\sigma (\tau,\phi(\tau,\xi ,t))$.\label{flow lemma5}
\item $\phi_\xi (\tau,\xi ,t)=v(t)$ solves the linear ordinary differential equation $$\dot v(t)=\sigma_\xi (\tau,\phi(\tau,\xi ,t))v(t)$$
 with initial condition $v(0)=1$ and so
 \begin{equation}\label{flow lemma phix}
 \phi_\xi (\tau,\xi ,t)=e^{\int_0^t\sigma_\xi (\tau,\phi(\tau,\xi ,s))\,ds}.
 \end{equation}\label{flow lemma6}
\item $\phi_\tau(\tau,\xi ,t)=w(t)$ solves the linear ordinary differential equation
 $$\dot w(t)=\sigma_\tau(\tau,\phi(\tau,\xi ,t))+\sigma_\xi (\tau,\phi(\tau,\xi ,t))w(t)$$
 with initial condition $w(0)=0$ and is hence given by
 \begin{equation}\label{flow lemma phitau}
 \phi_\tau(\tau,\xi ,t)=\int_0^te^{\int_s^t\sigma_\xi (\tau,\phi(\tau,\xi ,r))\,dr}\sigma_\tau(\tau,\phi(\tau,\xi ,s))\,ds .
 \end{equation}\label{flow lemma7}
 \item $\phi_t(\tau,\xi ,-t)=\phi_\xi (\tau,\xi ,-t)\sigma(\tau,\xi )$.\label{flow lemma8}
 \item $\phi_{\xi \xi }(\tau,\xi ,-t)\sigma(\tau,\xi )^2-2\phi_{\xi t}(\tau,\xi ,-t)\sigma(\tau,\xi )+\phi_{tt}(\tau,\xi ,-t)=-\phi_\xi (\tau,\xi ,-t)\phi_{tt}(\tau,\phi(\tau,\xi ,-t),t)$.\label{flow lemma9}
 \end{enumerate}
\end{lemma}

\begin{proof} Since $\sigma_\xi $ is bounded by assumption, $\sigma(\tau,\xi )$ satisfies both a linear-growth and a Lipschitz condition in $\xi $. Therefore  the ordinary differential equation $\dot y(t)=\sigma(\tau,y(t))$ admits a unique global solution for all initial values $y(0)$ and all $\tau\in[0,1]$. This implies assertions~\ref{flow lemma1} and~\ref{flow lemma4}.

To show the remaining assertions,
we introduce a two-dimensional extension of $\sigma$ by letting $\bm\sigma(t,\xi ):=(0,\sigma(t,\xi ))^\top$. Then the solution $\bm y(t)$ of the two-dimensional autonomous ordinary differential equation $\dot{\bm y}(t)=\bm\sigma(\bm y(t))$ with initial condition $\bm y(0)=(\tau,\xi )^\top$ is given by $(\tau,y(t))^\top $, where $y(t)$ is as above.  Thus, $\bm\phi(\tau,\xi ,t):=(\tau,\phi(\tau,\xi ,t))^\top$ is equal to the flow of the autonomous equation $\dot{\bm y}(t)=\bm\sigma(\bm y(t))$.
In view of~\ref{flow lemma1}, assertions~\ref{flow lemma2},~\ref{flow lemma3},~\ref{flow lemma6}, and~\ref{flow lemma7} therefore follow from Theorems 2.10 and 6.1 in~\citep{Teschl}. Assertion~\ref{flow lemma5}  follows by applying~\ref{flow lemma4} twice.

To prove~\ref{flow lemma8},  let $y:=\phi(\tau,\xi ,-t)$ so that $\xi =\phi(\tau,y,t)$ and $\phi(\tau,\phi(\tau,y,t),-t)=y$ by~\ref{flow lemma3}.  It follows that
\begin{equation}\label{flow lemma 8 aux eqn}
0=\frac\partial{\partial t}\phi(\tau,\phi(\tau,y,t),-t)=\phi_\xi (\tau,\phi(\tau,y,t),-t)\phi_t(\tau,y,t)-\phi_t(\tau,\phi(\tau,y,t),-t)
\end{equation}
Inserting $\xi =\phi(\tau,y,t)$  and using~\ref{flow lemma4} yields~\ref{flow lemma8}.

To prove~\ref{flow lemma9} we let once again $y:=\phi(\tau,\xi ,-t)$  and take the derivative of~\eqref{flow lemma 8 aux eqn} with respect to $t$. This yields
\begin{align*}0&=\phi_{\xi \xi }(\tau,\phi(\tau,y,t),-t)\phi_t(\tau,y,t)^2-\phi_{t\xi }(\tau,\phi(\tau,y,t),-t)\phi_t(\tau,y,t)\\
&+\phi_\xi (\tau,\phi(\tau,y,t),-t)\phi_{tt}(\tau,y,t)+\phi_{tt}(\tau,\phi(\tau,y,t),-t)-\phi_{t\xi }(\tau,\phi(\tau,y,t),-t)\phi_t(\tau,y,t).
\end{align*}
Using again $\xi =\phi(\tau,y,t)$  and~\ref{flow lemma4} gives
$$\phi_{\xi \xi }(\tau,\xi ,-t)\sigma(\tau,\xi )^2-2\phi_{\xi t}(\tau,\xi ,-t)\sigma(\tau,\xi )+\phi_{tt}(\tau,\xi ,-t)=-\phi_\xi (\tau,\xi ,-t)\phi_{tt}(\tau,\phi(\tau,\xi ,t),-t).
$$
Assertion~\ref{flow lemma9} will thus follow if we can show that
\begin{equation}\label{flow lemma9 aux eq 1}
\phi_{tt}(\tau,\phi(\tau,\xi ,t),-t)=\phi_{tt}(\tau,\phi(\tau,\xi ,-t),t).
\end{equation}
By~\ref{flow lemma5} and~\ref{flow lemma3}, the left-hand side of~\eqref{flow lemma9 aux eq 1}
 is equal to
 $$\sigma_\xi (\tau,\phi(\tau,\phi(\tau,\xi ,t),-t))\sigma(\tau,\phi(\tau,\phi(\tau,\xi ,t),-t))=\sigma_\xi (\tau,\xi )\sigma(\tau,\xi ),
 $$
 and the same argument gives that also the right-hand side of~\eqref{flow lemma9 aux eq 1}
is equal to $\sigma_\xi (\tau,\xi )\sigma(\tau,\xi )$. This implies~\eqref{flow lemma9 aux eq 1} and in turn~\ref{flow lemma9}.\end{proof}

\medskip

\begin{proof}[Proof of Theorem~\ref{general IDE existence thm}]
Since $\sigma_\xi $
is bounded by assumption, it follows from~\eqref{flow lemma phix} that there are constants $c,\varepsilon>0$ such that $\varepsilon\le \phi_\xi (\tau,\xi ,t)\le c$ for all $\tau$, $\xi $, and $t$. In particular, $\phi$ satisfies a linear growth condition in its second argument. It   follows that
$$g_1(t,y):=\frac{b(t,\phi(t,y,x(t)))}{\phi_\xi (t,y,x(t))}
$$
is continuous in $t$ and satisfies a linear growth condition in $y$, uniformly in $t\in[0,1]$. Since $\phi\in C^2(I\times\mathbb{R}\times\mathbb{R})$ and $b$ satisfies a local Lipschitz condition,  $g_1$ also satisfies a local Lipschitz condition uniformly in $t\in [0,1]$.  Next, we consider
$$g_2(t,y):=-\frac{\phi_\tau(t,y,x(t))}{\phi_\xi (t,y,x(t))}.
$$
It follows from~\eqref{flow lemma phitau} that the numerator is bounded in $y\in\mathbb{R}$, uniformly in $t\in[0,1]$. Moreover, $g_2(t,y)$ satisfies a local Lipschitz condition as $\phi\in C^2(I\times\mathbb{R}\times\mathbb{R})$. We now consider
$$g_3(t,y):=-\frac12\frac{\phi_{tt}(t,y,x(t))}{\phi_\xi (t,y,x(t))}.
$$
Since both $\phi$ and $\sigma$ satisfy linear growth conditions in their second arguments and $\sigma_\xi $ is bounded, it follows from part (e) of Lemma~\ref{flow lemma} that $g_3(t,y)$ satisfies a linear growth condition in $y$, uniformly in $t\in[0,1]$. As moreover $\phi_t\in C^2(I\times\mathbb{R}\times\mathbb{R})$, it follows that $g_3(t,y)$ satisfies a local Lipschitz condition uniformly in $t\in[0,1]$.
{When letting $A_1(t):=A(t)$, $A_2(t):=t$, and $A_3(t):=\<x\>_t$, we see that}  the Stieltjes integral equation~\eqref{Doss-Sussman Stieltjes int eq} satisfies the assumptions of Proposition~\ref{Peano prop}  so that~\eqref{Doss-Sussman Stieltjes int eq}  admits a unique solution $B$ for each initial value $y\in\mathbb{R}$. Using It\=o's formula as in the motivation of Theorem~\ref{general IDE existence thm} thus yields that $z(t):=\phi(t,B(t),x(t))$ is indeed a solution of~\eqref{general IDE existence thm eq in thm}. This establishes the existence of solutions.

To show uniqueness of solutions to~\eqref{general IDE existence thm eq in thm}, we let $\widetilde z$ be an arbitrary solution with initial condition $\widetilde z(0)=z_0$ and define
\begin{equation}\label{Doss B def eq}
\widetilde B(t):=\phi(t,\widetilde z(t),-x(t)).
\end{equation}
It follows from part (c)
of Lemma~\ref{flow lemma} that then $\phi(t,\widetilde B(t),x(t))=\widetilde z(t)$. We will show that $\widetilde B$ solves the Stieltjes integral equation~\eqref{Doss-Sussman Stieltjes int eq} and hence must coincide with $B$ due to the already established uniqueness of solutions to~\eqref{Doss-Sussman Stieltjes int eq}, {and then $\widetilde z(t)=z(t)$.} We clearly have $\widetilde B(0)=z_0$. To analyze the dynamics of $\widetilde B$, we want to apply It\=o's formula. To this end, we note first that, by definition, $\sigma(t,\widetilde z(t))$ is an admissible integrand for $x$ and that ${\<\widetilde  z\>_t}=\int_0^t\sigma(s,\widetilde z(s))^2\,d\<x\>_s$ as well as $\<\widetilde z,x\>_t=\int_0^t\sigma(s,\widetilde z(s))\,d\<x\>_s$ by  Proposition~\ref{ide sol qv prop} and polarization~\eqref{polarization in approx}. Thus,  It\=o's formula yields that
\begin{align}
\lefteqn{\widetilde B(t)-\widetilde B(0)}\nonumber\\
&=\int_0^t\phi_\tau(s,\widetilde z(s),-x(s))\,ds +\int_0^t\phi_\xi (s,\widetilde z(s),-x(s))\,d\widetilde z(s)-\int_0^t\phi_t(s,\widetilde z(s),-x(s))\,dx(s)\nonumber\\
&\qquad+\frac12\int_0^t\phi_{\xi \xi }(s,\widetilde z(s),-x(s))\,d\<\widetilde z\>_s+\frac12\int_0^t\phi_{tt}(s,\widetilde z(s),-x(s))\,d\<x\>_s\label{Doss first Ito appl eq}\\
&\qquad-\int_0^t\phi_{\xi t}(s,\widetilde z(s),-x(s))\,d\<\widetilde z,x\>_s.\nonumber
\end{align}
Applying the fact that $\widetilde z$ solves~\eqref{general IDE existence thm eq in thm},  the associativity theorems for Stieltjes and It\=o integrals,~\citep[Theorem I.5c]{WidderLaplaceTransform} and~\citep[Theorem 13]{SchiedCPPI}, and part~\ref{flow lemma8} of Lemma~\ref{flow lemma} give that
\begin{align*}
\lefteqn{\int_0^t\phi_\xi (s,\widetilde z(s),-x(s))\,d\widetilde z(s)}\\
&=\int_0^t\phi_\xi (s,\widetilde z(s),-x(s))\sigma(s,\widetilde z(s))\,dx(s)+\int_0^t\phi_\xi (s,\widetilde z(s),-x(s))b(s,\widetilde z(s))\,dA(s)\\
&=\int_0^t\phi_t(s,\widetilde z(s),-x(s))\,dx(s)+\int_0^t\phi_\xi (s,\widetilde z(s),-x(s))b(s,\widetilde z(s))\,dA(s).
\end{align*}
In particular in~\eqref{Doss first Ito appl eq} all It\=o integrals with respect to $dx(s)$ cancel out. Next, the sum of the integrals involving $d\<\widetilde z\>_s$, $d\<x\>_s$, or $d\<\widetilde z,x\>_s$ in~\eqref{Doss first Ito appl eq} is equal to
\begin{align*}\frac12\int_0^t\bigg(\phi_{\xi \xi }(s,\widetilde z(s),-x(s))\sigma(s,\widetilde z(s))^2+\phi_{tt}(s,\widetilde z(s),-x(s))-2\phi_{\xi t}(s,\widetilde z(s),-x(s))\sigma(s,\widetilde z(s))\bigg)\,d\<x\>_s\\
=-\frac12\int_0^t\phi_{\xi }(s,\widetilde z(s),-x(s))\phi_{tt}(s,\widetilde B(s),x(s))\,d\<x\>_s,
\end{align*}
where we have used part~\ref{flow lemma9} of Lemma~\ref{flow lemma} and~\eqref{Doss B def eq}.

Next, differentiating the identity $\xi =\phi(\tau,\phi(\tau,\xi ,t),-t)$ with respect to $\xi $ and $\tau$ yields
$$
\phi_\xi (\tau, \phi(\tau,\xi ,t),-t)=\frac1{\phi_\xi (\tau,\xi ,t)}
$$
and
$$\phi_\tau(\tau, \phi(\tau,\xi ,t),-t)=-\phi_\xi (\tau, \phi(\tau,\xi ,t),-t)\phi_\tau(\tau,\xi ,t)=-\frac{\phi_\tau(\tau,\xi ,t)}{\phi_\xi (\tau,\xi ,t)}.
$$
Since $\phi(s,\widetilde B(s),x(s))=\widetilde z(s)$ we hence get
$$\phi_{\xi }(s,\widetilde z(s),-x(s))=\frac1{\phi_\xi (s,\widetilde B(s),x(s))},\qquad \phi_\tau(s,\widetilde z(s),-x(s))=-\frac{{\phi_\tau(s,\widetilde B(s),x(s))}}{\phi_\xi (s,\widetilde B(s),x(s))}.
$$
Putting all this back into~\eqref{Doss first Ito appl eq} yields that
$$\widetilde B(t)-{z_0}=-\int_0^t\frac{{\phi_\tau(s,\widetilde B(s),x(s))}}{\phi_\xi (s,\widetilde B(s),x(s))}\,ds+\int_0^t\frac{b(s,\widetilde z(s))}{\phi_\xi (s,\widetilde B(s),x(s))}\,dA(s)-\frac12\int_0^t\frac{\phi_{tt}(s,\widetilde B(s),x(s))}{\phi_{\xi}(s,\widetilde B(s),x(s))}\,d\<x\>_s.
$$
That is, $\widetilde B$ solves~\eqref{Doss-Sussman Stieltjes int eq}.
 \end{proof}

\begin{proof}[Proof of Corollary~\ref{support corollary}] (a): As observed in the proof of Theorem~\ref{general IDE existence thm}, the derivative $\phi_\xi$ is bounded away from zero and from above. It is therefore sufficient to show that for every $\beta\in\mathbb{R}$ there exists $b\in\mathbb{R}$ such that the {solution} of the integral equation~\eqref{Doss-Sussman Stieltjes int eq} with constant {term} $b\in\mathbb{R}$ is such that $B(t_0)=\beta$.

Let us denote the solution {of~\eqref{Doss-Sussman Stieltjes int eq} with given $b\in\mathbb{R}$} by $B_b$. Since $\<x\>_t=t$ and $A(t)=t$, the equation~\eqref{Doss-Sussman Stieltjes int eq} is in fact an ordinary differential equation of the form
\begin{equation*}
B'_b(t)=b g(t,B_b(t))+f(t,B_b(t)),
\end{equation*}
where $f$ and $g$ are continuous and satisfy local Lipschitz conditions in $\xi$. In addition, $g>0$ is bounded and bounded away from zero, and $f(t,\xi) $  has at most linear growth in $\xi$. A standard argument using Gronwall's inequality therefore yields the continuity of the map  $b\mapsto B_b(t_0)$. Moreover, there are constants $c_g^\pm$ and $c_f^\pm$ such that {$c_g^\pm>0$ and}
\begin{align*}
B'_b(t)&\le bc_g^++c_f^+B_b(t),\\
B'_b(t)&\ge  bc_g^-+c_f^-B_b(t).
\end{align*}
A standard comparison result for ordinary differential equations~\citep[Theorem 1.3]{Teschl} yields that $B_b(t)$ is bounded from above and from below by the respective solutions of the ordinary differential equations
$$y'(t)=bc_g^\pm+c_f^\pm y(t),\qquad y(0)=z_0.
$$
 Since the values of these solutions at $t_0$ range  through  all of $\mathbb{R}$ as $b$  varies between  $-\infty$ and $+\infty$, it follows that $\inf_{b\in\mathbb{R}}B_b(t_0)=-\infty$ and $\sup_{b\in\mathbb{R}}B_b(t_0)=+\infty$. The already established continuity of $b\mapsto B_b(t_0)$ therefore yields the result. 

(b): Let $U$ be an open subset of $C[0,1]$ and take $z_0:=f(0)$ for some $f\in U$. Since $\phi_\xi$ is bounded away from zero and from above, it is not difficult to construct a continuously differentiable function $B$ such that $z(t):=\phi(t,B(t),x(t))\in U$ for all $t$.  But when letting
$$b(t):=\phi_\xi(t,B(t),x(t))B'(t)+\phi_\tau(t,B(t),x(t))+\frac12\phi_{tt}(t,B(t),x(t)),
$$
it follows that $B$ solves~\eqref{Doss-Sussman Stieltjes int eq}  for this choice of $b$. Hence, $z$ solves~\eqref{z dense IDE}, {which completes the proof of part (b)}.

{The final assertion of the corollary follows from the fact that the solutions of~\eqref{z dense IDE 1} and~\eqref{z dense IDE} belong to the  set~\eqref{Bick Willinger set} according to to Lemma~\ref{zero qv lemma}.} \end{proof}

\noindent{{\bf Acknowledgement.} The authors thank two anonymous referees for comments that helped to improve a previous version of the manuscript. }

	 \parskip-0.5em\renewcommand{\baselinestretch}{0.9}\normalsize
\bibliography{CTbook}{}
\bibliographystyle{abbrv}
%\bibitem {}

\end{document}